\documentclass[]{article}

\usepackage{amssymb,amsfonts,amsmath,amsthm}
\usepackage{graphicx}
\usepackage{natbib}
\usepackage[dvipsnames]{xcolor}
\usepackage{float}
\usepackage{natbib}
\usepackage[none]{hyphenat}
\usepackage{geometry}
\usepackage{setspace}
\usepackage{booktabs}
\usepackage{microtype}

\newtheorem{theorem}{Theorem}[section]
\newtheorem*{theoremA}{Theorem A}
\newtheorem*{theoremB}{Theorem B}
\newtheorem*{theoremC}{Theorem C}

\newtheorem{proposition}[theorem]{Proposition}
\newtheorem{lemma}[theorem]{Lemma}

\newtheorem{example}[theorem]{Example}
\newtheorem{remark}[theorem]{Remark}

\title{Invariant theory for non-reductive actions: extensions of Hilbert and Schwarz theorems} 

\author{Leandro Nery\\
	{\small Department of Mathematics, CCET - Federal University of S\~ao Carlos}\\
	{\small C.P. 676, 13565-905 S\~ao Carlos SP, Brazil \footnote{Email address: leandro.oliveira@ufscar.br}} }
\date{}

\begin{document}

\maketitle

\begin{abstract}
Classical invariant theory establishes a systematic correspondence between algebraic and smooth invariants for compact and reductive Lie groups. However, the extension of these results to non-compact and non-reductive regimes remains a subject of ongoing research. This paper examines the divergence between the algebras of polynomial and smooth invariants in two specific settings: discrete subgroups of the Lorentz group $O(n,1)$ acting on $\mathbb{R}^{n,1}$, and cocompact actions on smooth manifolds. We prove that for discrete Lorentz groups, the ring of polynomial invariants is finitely generated, but the smooth invariants are not generated by the polynomial ones. In the case of cocompact actions, we demonstrate that the polynomial invariant ring reduces to constants, while the algebra of smooth invariants is finitely generated and determined by the smooth structure of the quotient manifold. These results lead to a classification of invariant-theoretic regimes into four categories, identifying the boundaries of the Hilbert--Weyl and Schwarz theorems and establishing the role of properness in the alignment of algebraic and analytic descriptions of symmetry.
\end{abstract}

\bigskip


\bigskip

\onehalfspace

\tableofcontents

\section{Introduction}

Classical invariant theory focuses primarily on the actions of compact and reductive Lie groups. A fundamental result in this setting is the Hilbert-Weyl theorem, which establishes that, for these classes of groups, the ring of polynomial invariants is finitely generated \cite{hilbert1890,hilbert1893,weyl1946}. This result provided a positive answer to the 14th problem of Hilbert for the reductive case. While Hilbert's original work addressed these groups, the question of whether finite generation holds for any group acting linearly on a polynomial ring remained open until Nagata provided a negative answer for a non-reductive group \cite{nagata1959}.\\

In the reductive case, the algebraic structure of polynomial invariants determines the associated analytic theory. Results by Schwarz and Luna established that any smooth invariant function can be expressed as a smooth function of the polynomial generators \cite{luna1976, schwarz1975}. This correspondence implies that the properties of smooth invariants are characterized by the polynomial ones.\\

However, beyond the reductive and compact settings, this relationship between algebra and analysis can diverge. This paper investigates two distinct non-compact, non-reductive regimes where the classical correspondence between the Hilbert-Weyl and Schwarz theorems is altered. We show that the setion of these two structures reveals a more complex landscape than the one observed in the reductive case.\\

First, we examine the case of discrete Lorentz groups acting on Minkowski space. We establish that, for these groups, the polynomial invariant ring remains finitely generated (Theorem A). This provides a class of non-reductive groups where a positive answer to Hilbert's 14th problem persists. Nevertheless, we demonstrate that the smooth theory diverges from the algebraic one: the classical Schwarz theorem fails, as the polynomial invariants no longer generate the algebra of smooth invariants (Proposition 4.1).\\

Second, we analyze cocompact actions of discrete groups on $\mathbb{R}^n$. In this regime, an opposite phenomenon occurs. The ring of polynomial invariants collapses to the constants (Theorem B), yielding a case where the 14th problem is satisfied in a degenerate sense. Despite this algebraic collapse, a structured smooth invariant theory remains, and we prove a smooth analogue of the Schwarz theorem where the algebra of smooth invariants is finitely generated by a set of basic smooth functions (Theorem C).\\

These results suggest a broader framework for invariant theory where the classical reductive case is seen as a specific, well-behaved regime. By investigating the Lorentz and cocompact settings, we observe that the failure of either the Hilbert-Weyl or the Schwarz theorem does not imply a lack of structure. Instead, it represents a shift in the nature of that structure, which may reside in the algebraic properties of the orbit space or in the geometric constraints of the action. This fourfold classification, further detailed in Section 6, provides a basis for investigating other non-reductive settings and delineates the conditions under which the analytic and algebraic aspects of invariant theory diverge.\\

The paper is organized as follows. Section 2 presents the necessary background on the Lorentz group and hyperbolic complex numbers. Section 3 establishes the finite generation of polynomial invariants for the discrete Lorentz case (Theorem A). Section 4 demonstrates the failure of the classical Schwarz theorem in this setting. Section 5 analyzes the cocompact regime, proving the collapse of polynomial invariants and the corresponding generation of smooth invariants (Theorems B and C). Finally, Section 6 provides a discussion of the resulting conceptual framework.

\subsection{Statement of main results}

We now formally state the main results of the paper.

\subsubsection{The Lorentz regime: polynomial generation and the failure of the Schwarz property}

This work investigates two distinct non-reductive and non-compact contexts to characterize the relationship between algebraic and smooth invariants. First, we examine discrete subgroups of the Lorentz group $O(n,1)$. Consider the hyperbolic rotation $\mathcal{H}_\beta \in O(1,1)$ given by
\begin{eqnarray} \label{matrizhb}
	\mathcal{H}_\beta = \begin{bmatrix}
		\cosh \beta & \sinh \beta \\
		\sinh \beta & \cosh \beta 
	\end{bmatrix}.
\end{eqnarray}

The following result was introduced in \cite{manoel2025affine} within the context of centrosymmetric matrices. We provide its proof in Section 3 as a basis for the higher-dimensional cases considered in this work.

\noindent{\bf Theorem 3.1}\cite{manoel2025affine}
	{\it Let $\Gamma$ be a discrete subgroup of $O(1,1)$ generated by a hyperbolic rotation $\mathcal{H}_\beta$, with fixed $\beta \in \mathbb{R} \setminus \{0\}$, acting on the Minkowski space $\mathbb{R}^{1,1}$. Then the ring of invariants $\mathcal{P}(\Gamma)$ is finitely generated and}
	$$\mathbb{R}[x,y]^\Gamma = \mathbb{R}[x^2 - y^2].$$

This result is extended to higher dimensions by considering the action of a discrete Lorentz group $O(n,1)$ on Minkowski space $\mathbb{R}^{n,1}$
 Consider the matrix
$$
\mathcal{L}_{\alpha,\beta} = \begin{bmatrix} R_{\alpha} & 0 \\ 0 & H_{\beta} \end{bmatrix},
$$
where $R_{\alpha} \in O(n-1)$ is a rotation (or reflection) in the Euclidean subspace $\mathbb{R}^{n-1}$, and $ \mathcal{H}_\beta \in O(1,1)$ is a hyperbolic matrix acting on the plane $(x_n, y)$, with fixed parameters $\alpha \in [0,2\pi)$ and $\beta \in \mathbb{R} \setminus \{0\}$.

\begin{theoremA}\label{theoremaB}
		Let $\Gamma$ be a discrete subgroup of $O(n,1)$ generated by a matrix $A$ conjugate in $O(n,1)$ to a block matrix of the form
	$$
	\mathcal{L}_{\alpha,\beta} = 
	\begin{bmatrix}
		R_\alpha & 0 \\
		0 & \mathcal{H}_\beta
	\end{bmatrix},
	$$
	with $R_\alpha \in O(n-1)$ and fixed $\beta \in \mathbb{R}\setminus\{0\}$. Then the ring of invariant polynomials
	$
	\mathcal{P}_{\mathbb{R}^{n,1}}(\Gamma)
	$
	admits a Hilbert basis given by the union of a Hilbert basis of $\mathcal{P}_{\mathbb{R}^{n-1}}(\langle R_\alpha \rangle)$ with the set $\{x^2 - y^2\}$.
\end{theoremA}

In contrast to the reductive setting, these finitely generated polynomial invariants do not determine the smooth invariant theory.

\noindent{\bf Proposition 4.1}(Failure of the Schwarz Theorem for $O(1,1)$)
	\label{prop:lorentz_schwarz_fails}
	 {\it Let $\Gamma = \langle \mathcal{H}_\beta \rangle$ be the group from Theorem~\ref{teorema31}, and let $\rho(x,y) = x^2 - y^2$ be the unique non-trivial polynomial invariant. Then the inclusion $\rho^*\mathcal{E}(\mathbb{R}) \subset \mathcal{E}(\mathbb{R}^2)^\Gamma$ is proper, i.e.,
	\[
	\rho^*\mathcal{E}(\mathbb{R}) \subsetneq \mathcal{E}(\mathbb{R}^2)^\Gamma.
	\]
}

\subsubsection{The cocompact regime: polynomial triviality and smooth generation}

In a second, complementary direction, we analyze cocompact actions of discrete groups on $\mathbb{R}^n$. In this setting, the relationship between algebraic and smooth invariants is reversed compared to the Lorentz case. The ring of polynomial invariants reduces to the constants, which represents a case where the classical Hilbert-Weyl theorem provides no non-trivial information. However, a theory of smooth invariants remains, as we establish through the following results:

\begin{theoremB}
	Let \( \Gamma \) be a discrete group acting freely, properly, and cocompactly on \( \mathbb{R}^n \). Then the algebra of polynomial invariants is trivial, i.e., \( P(\mathbb{R}^n)^\Gamma = \mathbb{R} \).
\end{theoremB}

Despite the absence of non-constant polynomial invariants, the smooth invariant theory remains finitely generated. This is established by Theorem C, which provides an analogue of the Schwarz theorem for the cocompact regime:

\begin{theoremC}[Smooth Invariants for Cocompact Actions]
	Let \( \Gamma \) be a discrete group acting freely, properly, and cocompactly on a smooth manifold \( M \). Let \( \pi: M \to M/\Gamma \) be the projection map.
	Then:
	\begin{enumerate}
		\item
		\[
		\pi^*: \mathcal{E}(M/\Gamma) \xrightarrow{\cong} \mathcal{E}(M)^\Gamma
		\]
		is a topological isomorphism of Fréchet spaces.
		\item There exist smooth invariants \( \sigma_1, \ldots, \sigma_m \in \mathcal{E}(M)^\Gamma \) such that
		\[
		\sigma^* \mathcal{E}(\mathbb{R}^m) = \mathcal{E}(M)^\Gamma,
		\]
		where \( \sigma = (\sigma_1, \ldots, \sigma_m): M \to \mathbb{R}^m \).
		\item For any manifold \( X \) with trivial \( \Gamma \)-action,
		\[
		(\mathrm{id} \times \sigma)^* \mathcal{E}(X \times \mathbb{R}^m) = \mathcal{E}(X \times M)^\Gamma.
		\]
	\end{enumerate}
	
\end{theoremC}

These results suggest that the relationship between algebraic and smooth invariants is more diverse than the one observed in the reductive case. The findings presented here lead to a fourfold classification based on the interplay of algebra, analysis, and geometry, delineating the boundaries of the classical Hilbert–Weyl and Schwarz theorems. This framework, discussed in detail in Section 6, provides a basis for investigating invariant theory in other non-reductive settings.

\section{Preliminaries: The Lorentz group and its discrete subgroups}

\subsection*{Minkowski space and the Lorentz group}

The Minkowski space of signature \((n,1)\), denoted by \(\mathbb{R}^{n,1}\), is identified with \(\mathbb{R}^{n+1}\) endowed with the quadratic form
\[
Q(x_1, \dots, x_n, y) = x_1^2 + \cdots + x_n^2 - y^2.
\]
The Lorentz group \(O(n,1)\) is defined as the group of invertible linear transformations of \(\mathbb{R}^{n+1}\) that preserve \(Q\), that is,
\[
O(n,1) := \left\{A \in \operatorname{GL}(n+1,\mathbb{R}) \mid Q(Ax) = Q(x) \text{ for all } x \in \mathbb{R}^{n+1} \right\}.
\]

The geometric structure preserved by \(O(n,1)\) and the properties of its discrete subgroups are central to our analysis. For a detailed treatment of the Lorentz group and the geometry of its actions in the context of hyperbolic space, we refer to \cite{ratcliffe2019}.

It is well known that any matrix in $O(1,1)$ can be written in the form
$$
\begin{bmatrix}
	\epsilon & 0 \\
	0 & \varepsilon
\end{bmatrix}
\begin{bmatrix}
	\cosh \beta & \sinh \beta \\
	\sinh \beta & \cosh \beta
\end{bmatrix}
=
\begin{bmatrix}
	\epsilon & 0 \\
	0 & \varepsilon
\end{bmatrix} \mathcal{H}_\beta,
$$
where $\epsilon, \varepsilon \in \{\pm1\}$, $\beta \in \mathbb{R}$, and $\mathcal{H}_\beta$ denotes the standard hyperbolic rotation (see~\eqref{matrizhb}). Similarly, every matrix $A \in O(n,1)$ admits a factorization of the form
$$
A =
\begin{bmatrix}
	S & 0 \\
	0 & \epsilon
\end{bmatrix}
\begin{bmatrix}
	I_{n-1} & 0 \\
	0 & \mathcal{H}_\beta
\end{bmatrix}
\begin{bmatrix}
	R & 0 \\
	0 & 1
\end{bmatrix},
$$
where $S \in O(n)$, $R \in SO(n)$, $\epsilon \in \{\pm1\}$, and $\beta \in \mathbb{R}$.

The Lorentz group $O(n,1)$ has four connected components. Let us define the matrices
$$
\Lambda_p = \begin{bmatrix}
	I_{n-1,1} & 0 \\
	0 & 1
\end{bmatrix}, \quad
\Lambda_t = J, \quad
\Lambda_{pt} = \Lambda_p \Lambda_t,
$$
where $I_{n-1,1}$ is the diagonal matrix with $n-1$ entries equal to $1$ and one equal to $-1$. Denoting by $SO_0(n,1)$ the identity component of $O(n,1)$, we have the disjoint decomposition
$$
O(n,1) = SO_0(n,1) \,\dot\cup\, \Lambda_p SO_0(n,1) \,\dot\cup\, \Lambda_t SO_0(n,1) \,\dot\cup\, \Lambda_{pt} SO_0(n,1).
$$

\subsection*{Hyperbolic complex numbers and the Minkowski plane}

The Minkowski plane $\mathbb{R}^{1,1}$ admits a algebraic model via the set
$$
\mathbb{D} := \{x + h y \mid x, y \in \mathbb{R},\ h^2 = 1,\ h \notin \mathbb{R}\},
$$
called the algebra of hyperbolic complex numbers (also called Lorentz or split-complex numbers). Although the algebraic structure of $\mathbb{D}$ is analogous to the complex plane, it is not an integral domain; indeed, $(1 + h)(1 - h) = 0$. 
Nevertheless, standard operations from complex analysis, such as conjugation, modulus, and argument, are well-defined in this setting.

The algebra $\mathbb{D}$ is used to characterize Lorentz transformations, light cone geometry, and causal structure in dimension two. Further geometric interpretations can be found in~\cite{gurses2015one, harkin2004geometry, yaglom2014complex}.

\subsection*{Discrete subgroups of the Lorentz group}

We focus on discrete subgroups $\Gamma \subset O(n,1)$ generated by matrices of the form
$$
\mathcal{L}_{\alpha,\beta} =
\begin{bmatrix}
	R_\alpha & 0 \\
	0 & \mathcal{H}_\beta
\end{bmatrix},
$$
where $R_\alpha \in O(n-1)$ is a Euclidean rotation (or reflection) in $\mathbb{R}^{n-1}$, and $\mathcal{H}_\beta \in O(1,1)$ is a hyperbolic boost matrix defined as in~\eqref{matrizhb}, with fixed $\beta \in \mathbb{R} \setminus \{0\}$.

\begin{remark}
	The subgroup $\Gamma = \langle \mathcal{H}_\beta \rangle \subset O(1,1)$ is an infinite discrete group, isomorphic to $\mathbb{Z}$, and is not reductive. The absence of $\Gamma$-invariant positive-definite inner products on $\mathbb{R}^2$, due to the hyperbolic nature of the generator ($\beta \neq 0$), implies that the action is not reductive (see \cite[p.~125]{Humphreys1975}).
\end{remark}

The discrete subgroups $\Gamma \subset O(n,1)$ considered here are constructed via compositions of Euclidean isometries and hyperbolic transformations. In the following sections, we analyze the invariant rings for these actions, examining the conditions for finite generation and the validity of the Schwarz theorem.

\section{Discrete Lorentz group in Minkowski space}

This section addresses the structure of invariant rings under the action of discrete subgroups of $O(n,1)$. We first examine the case of $O(1,1)$, where the geometry of the Minkowski plane allows for a direct description of polynomial invariants. These results are then extended to higher-dimensional Minkowski spaces.

\subsection{$\Gamma$ as a discrete subgroup of $O(1,1)$}

Consider the group $\Gamma = \langle \mathcal{H}_\beta \rangle$ acting on the Minkowski plane $\mathbb{R}^{1,1}$, where $\mathcal{H}_\beta$ is the hyperbolic rotation in (\ref{matrizhb}), with fixed
$ \beta \in \mathbb{R} \setminus \{0\}.$

\begin{theorem}\label{teorema31}
	Let $\Gamma$ be a discrete subgroup of $O(1,1)$ generated by a hyperbolic rotation $\mathcal{H}_\beta$ with fixed $\beta \in \mathbb{R} \setminus \{0\}$. Then the invariant ring $\mathcal{P}(\Gamma)$ is finitely generated and
	$$
	\mathbb{R}[x,y]^\Gamma = \mathbb{R}[x^2 - y^2].
	$$
\end{theorem}

\begin{proof}
	Consider the identification of the Minkowski plane with the algebra of hyperbolic numbers $\mathbb{D}$. Let $z = x+hy$ and its conjugate $\overline{z} = x-hy$. The action of the generator on $z$ is given by $\mathcal{H}_\beta z = e^{h\beta} z$. Any polynomial $f \in \mathbb{R}[x,y]$ can be uniquely expressed as a polynomial in the variables $z$ and $\overline{z}$ over $\mathbb{R}$,
	$$ f(z, \overline{z}) = \sum_{r,s} a_{rs} z^r \overline{z}^s. $$
	The condition for $\Gamma$-invariance, $f(e^{h\beta} z, e^{-h\beta} \overline{z}) = f(z, \overline{z})$, implies the identity
	$$ \sum_{r,s} a_{rs} e^{h\beta(r-s)} z^r \overline{z}^s = \sum_{r,s} a_{rs} z^r \overline{z}^s. $$
	Since the monomials $z^r \overline{z}^s$ are linearly independent over $\mathbb{R}$, we must have $a_{rs} (e^{h\beta(r-s)} - 1) = 0$ for all $r, s$. Given that $\beta \neq 0$ and $h^2=1$, the term $e^{h\beta(r-s)} - 1$ vanishes if and only if $r = s$. Consequently, $f$ is a polynomial in the product $z\overline{z} = x^2 - y^2$, which establishes that $\mathbb{R}[x,y]^\Gamma = \mathbb{R}[x^2 - y^2]$.
\end{proof}

\begin{remark}
	The subgroups $\langle \mathcal{H}_\beta \rangle$ and $\langle -\mathcal{H}_\beta \rangle$ are distinct in $O(1,1)$ but act identically on the ring of invariants. Since $-\mathcal{H}_\beta = \Lambda_{pt} \mathcal{H}_\beta$ and the element $\Lambda_{pt}$ acts as $-I$, it preserves any polynomial composed of even-degree monomials. Specifically, both subgroups preserve the quadratic form $Q(x,y) = x^2 - y^2$ and, consequently, share the same invariant ring $\mathbb{R}[x^2 - y^2]$. Therefore, Theorem~\ref{teorema31} also applies to subgroups generated by matrices in the component $\Lambda_{pt} SO_0(1,1)$.
\end{remark}

To complete the classification of discrete cyclic subgroups in $O(1,1)$, we consider the components $\Lambda_p SO_0(1,1)$ and $\Lambda_t SO_0(1,1)$. The matrices in these components take the form
$$\begin{bmatrix}
	\cosh \beta & \sinh \beta \\
	-\sinh \beta & -\cosh \beta
\end{bmatrix} \in \Lambda_p SO_0(1,1) \quad \text{and} \quad
\begin{bmatrix}
	-\cosh \beta & -\sinh \beta \\
	\sinh \beta & \cosh \beta
\end{bmatrix} \in \Lambda_t SO_0(1,1),$$
respectively. These matrices are involutive and generate finite subgroups isomorphic to $\mathbb{Z}_2$. Since these groups are finite, the action is reductive and the classical Hilbert-Weyl theorem applies.

\begin{proposition}
	Any such involutive element $A$ is conjugate to either the reflection $\kappa_x = \operatorname{diag}(1, -1)$ or $\kappa_y = \operatorname{diag}(-1, 1)$. 
\end{proposition}
The conjugation $M A M^{-1}$ is realized by the matrix$$M = \begin{bmatrix} \cosh \tfrac{\beta}{2} & -\sinh \tfrac{\beta}{2} \\ -\sinh \tfrac{\beta}{2} & \cosh \tfrac{\beta}{2} \end{bmatrix}.$$Consequently, the invariant rings for these cases are $\mathbb{R}[x, y^2]$ or $\mathbb{R}[y, x^2]$, following the standard theory of finite reflection groups.

\subsection{Discrete subgroups of $O(n,1)$}

The extension of the invariant theory from $O(1,1)$ to higher-dimensional Lorentz groups relies on the block-diagonal structure of the generators. We consider subgroups $\Gamma \subset O(n,1)$ generated by matrices that decouple the action into an orthogonal transformation on $\mathbb{R}^{n-1}$ and a hyperbolic boost on $\mathbb{R}^{1,1}$. This decomposition allows the use of the following result on product actions:

\begin{lemma}[\cite{baptistelli2016}, Lemma 3.3]\label{lemma3_3}
	Let $\Gamma_1$ and $\Gamma_2$ be groups acting linearly on vector spaces $X$ and $Y$, respectively. If $\{u_1, \dots, u_r\}$ is a Hilbert basis for the invariant ring $\mathcal{P}_X(\Gamma_1)$ and $\{v_1, \dots, v_s\}$ is a Hilbert basis for $\mathcal{P}_Y(\Gamma_2)$, then the union
	$$\{u_1, \dots, u_r, v_1, \dots, v_s\}$$
	is a Hilbert basis for the invariant ring of the product action, $\mathcal{P}_{X \times Y}(\Gamma_1 \times \Gamma_2)$.
\end{lemma}

As an application of Lemma \ref{lemma3_3}, we obtain the following characterization of the invariant rings for discrete subgroups of $O(n,1)$.

\begin{theoremA}
	Let $\Gamma$ be a discrete subgroup of $O(n,1)$ generated by a matrix $A \in O(n,1)$. 
	If $A$ is conjugate in $O(n,1)$ to a block matrix of the form
	$$
	\mathcal{L}_{\alpha,\beta} = 
	\begin{bmatrix}
		R_\alpha & 0 \\
		0 & \mathcal{H}_\beta
	\end{bmatrix},
	$$
	with $R_\alpha \in O(n-1)$ and fixed $\beta \in \mathbb{R}\setminus\{0\}$, then the ring of invariant polynomials
	$
	\mathcal{P}_{\mathbb{R}^{n,1}}(\Gamma)
	$
	admits a Hilbert basis given by the union of a Hilbert basis of $\mathcal{P}_{\mathbb{R}^{n-1}}(\langle R_\alpha \rangle)$ with the set $\{x^2 - y^2\}$.
\end{theoremA}
\begin{proof}
	By hypothesis, the action of $\Gamma$ is conjugate to the product action of $\Gamma_1 = \langle R_\alpha \rangle$ on $\mathbb{R}^{n-1}$ and $\Gamma_2 = \langle \mathcal{H}_\beta \rangle$ on $\mathbb{R}^{1,1}$. Since $\Gamma$ is generated by the block-diagonal matrix $\mathcal{L}_{\alpha,\beta}$, any polynomial $f \in \mathcal{P}(\mathbb{R}^{n,1})^\Gamma$ satisfies the product invariance condition. By Theorem \ref{teorema31}, the Hilbert basis for $\mathcal{P}(\mathbb{R}^{1,1})^{\Gamma_2}$ is the single generator $\rho(x,y) = x^2 - y^2$. Applying Lemma \ref{lemma3_3}, the Hilbert basis for the total space $\mathbb{R}^{n-1} \times \mathbb{R}^{1,1}$ is obtained by the union of the Hilbert bases of the respective factors. The result follows from the fact that $R_\alpha$ is a compact operator, ensuring that $\mathcal{P}(\mathbb{R}^{n-1})^{\Gamma_1}$ is finitely generated by the classical Hilbert-Weyl theorem.
\end{proof}

\medskip

The following example illustrates the application of Theorem A to a non-compact subgroup where a finite rotation interacts with a hyperbolic boost.

\begin{example}
	Let $\Gamma \subset O(3,1)$ be the infinite cyclic subgroup generated by the block matrix$$A = \begin{bmatrix}
		R_{2\pi/3} & 0 \\
		0 & \mathcal{H}_\beta
	\end{bmatrix} =
	\begin{bmatrix}
		\cos(\tfrac{2\pi}{3}) & -\sin(\tfrac{2\pi}{3}) & 0 & 0 \\
		\sin(\tfrac{2\pi}{3}) & \cos(\tfrac{2\pi}{3}) & 0 & 0 \\
		0 & 0 & \cosh\beta & \sinh\beta \\
		0 & 0 & \sinh\beta & \cosh\beta
	\end{bmatrix},$$with $\beta \neq 0$. The action of the upper block corresponds to the cyclic subgroup $C_3 \subset SO(2)$, whose invariant ring is generated by $\rho_1 = x^2+y^2$ and the cubic harmonics $\rho_2 = x^3 - 3xy^2$ and $\rho_3 = 3x^2y - y^3$. By Theorem A, a Hilbert basis for the $\Gamma$-action on $\mathbb{R}^{3,1}$ is given by the union of these generators with the Minkowski quadratic form of the hyperbolic block:$$\mathcal{B} = \{\, x^2+y^2,\; x^3 - 3xy^2,\; 3x^2y - y^3,\; z^2-w^2 \,\}.$$
\end{example}

\medskip

This construction identifies a specific class of non-compact discrete subgroups of $O(n,1)$ that admit finite Hilbert bases, extending the classical theory to a non-reductive Lorentzian setting through the decomposition provided by Theorem A.

\begin{remark}
	While Theorem A covers elements conjugate to a block-diagonal form, $O(n,1)$ also contains parabolic elements. These elements are characterized by a unipotent structure and, despite generating discrete subgroups, do not admit the product-block decomposition required by Lemma~\ref{lemma3_3}. In such cases, the determination of a Hilbert basis relies on the invariant theory of unipotent groups, and the corresponding smooth-algebraic contrast for these actions remains a subject for future investigation.
\end{remark}

\section{Smooth invariants and the role of compactness}

The Schwarz Theorem is a fundamental result in the invariant theory of compact Lie groups. Let $\mathcal{E}(\mathbb{R}^n)$ denote the algebra of real-valued smooth functions on $\mathbb{R}^n$ endowed with the topology of uniform convergence of all derivatives on compact sets. 

The theorem states that for an orthogonal action of a compact group $\Gamma$ on $\mathbb{R}^n$, the algebra of smooth invariants $\mathcal{E}(\mathbb{R}^n)^\Gamma$ is generated by polynomial invariants. Specifically, if $\{\rho_1, \ldots, \rho_k\}$ is a Hilbert basis for the polynomial ring $\mathcal{P}[\mathbb{R}^n]^\Gamma$, then the pullback map $\rho^*$ induces an isomorphism of algebras:
$$\rho^* : \mathcal{E}(\mathbb{R}^k) \xrightarrow{\cong} \mathcal{E}(\mathbb{R}^n)^\Gamma,$$
where $\rho^*(F) = F(\rho_1, \ldots, \rho_k)$. This result provides a bridge between algebraic and differential invariant theory, with extensive applications in singularity theory and the study of equivariant maps.

This raises a fundamental question regarding the necessity of compactness in Schwarz's formulation. Specifically, one may ask whether a similar correspondence between smooth and algebraic invariants persists for non-compact groups that possess a rigid geometric structure, such as discrete subgroups of the Lorentz group.

Let $\Gamma = \langle \mathcal{H}_\beta \rangle \subset O(1,1)$ be the discrete subgroup generated by a hyperbolic rotation $\mathcal{H}_\beta$ acting on $\mathbb{R}^{1,1}$. Despite its discrete nature and clear geometric action, we demonstrate that the absence of compactness leads to three fundamental obstructions that prevent a direct adaptation of Schwarz's Theorem.

To analyze these obstructions, it is convenient to employ light-cone coordinates
\[
u = x + y, \quad v = x - y.
\]
In this coordinate system, the invariant quadratic form becomes $Q(u,v) = uv$, and the generator acts as $(u, v) \mapsto (u e^\beta, v e^{-\beta})$. Consequently, for any $t \neq 0$, the hyperbola $uv = t$ consists of two disconnected branches, each preserved by the action of $\Gamma$. As we shall demonstrate, this topological separation underlies the failure of the smooth extension property in the non-compact setting.

\bigskip

\subsection{First obstruction: Insufficiency of polynomial invariants}

The first obstruction to a smooth Schwarz-type theorem in the Lorentz setting is the existence of smooth invariants that cannot be expressed as functions of the fundamental polynomial invariants. This phenomenon arises directly from the disconnected nature of the orbits in $\mathbb{R}^{1,1}$

\begin{proposition} \label{prop:insufficiency}
	Let $\Gamma = \langle \mathcal{H}_\beta \rangle$ be the discrete subgroup of $O(1,1)$ generated by a hyperbolic rotation. Let $\rho(x,y) = x^2 - y^2$ be the generator of the invariant ring $\mathcal{P}(\mathbb{R}^{1,1})^\Gamma$. Then the pullback map $\rho^* : \mathcal{E}(\mathbb{R}) \to \mathcal{E}(\mathbb{R}^2)^\Gamma$ is not surjective; that is,
	\[
	\rho^*\mathcal{E}(\mathbb{R}) \subsetneq \mathcal{E}(\mathbb{R}^2)^\Gamma.
	\]
\end{proposition}

\begin{proof}
	We construct a smooth $\Gamma$-invariant function $G$ that is not a function of the quadratic form $\rho$. Let $G: \mathbb{R}^2 \to \mathbb{R}$ be defined by
	\[
	G(x,y) =
	\begin{cases}
		\exp\left(-\frac{1}{x^2 - y^2}\right), & \text{if } x^2 - y^2 > 0 \text{ and } x > 0, \\0, & \text{otherwise}.\end{cases}
	\]
	To verify that $G \in \mathcal{E}(\mathbb{R}^2)$, note that $G$ is clearly smooth on the open regions $\{x^2 - y^2 > 0, x > 0\}$ and its complement. Along the light-cone $x^2 - y^2 = 0$, the function and all its partial derivatives vanish as $t = \rho(x,y) \to 0^+$ due to the rapid decay of the exponential term, ensuring $C^\infty$ regularity across the boundary.
	
	The $\Gamma$-invariance of $G$ is best seen in light-cone coordinates $u = x+y$ and $v = x-y$, where the action is $(u,v) \mapsto (ue^\beta, ve^{-\beta})$. The region $\rho(x,y) > 0$ comprises two disjoint components: the right branch ($u, v > 0$) and the left branch ($u, v < 0$). Since the action of $\mathcal{H}_\beta$ preserves the quadrants, each branch is an invariant manifold. On the right branch, $G(u,v) = \exp(-1/uv)$, which is invariant under the scaling $u \mapsto ue^\beta$ and $v \mapsto ve^{-\beta}$. Since $G$ vanishes identically on the left branch and in the region $uv \le 0$, it follows that $G \in \mathcal{E}(\mathbb{R}^2)^\Gamma$.
	
	Finally, suppose there exists $f \in \mathcal{E}(\mathbb{R})$ such that $G = f \circ \rho$. This would imply that $G$ is constant on the level sets of $\rho$. However, for any fixed $t > 0$, the level set $\rho^{-1}(t)$ contains points in both the right and left branches. On the right branch, $G$ takes the value $e^{-1/t} > 0$, while on the left branch, $G$ is identically zero. This contradiction shows that $G \notin \rho^*\mathcal{E}(\mathbb{R})$, completing the proof.
\end{proof}

\subsection{Second obstruction: Failure of orbit separation}
	In the compact setting, a necessary condition for the Schwarz Theorem is that the Hilbert map $\rho$ must separate the orbits of the group action. For the Lorentz group action, however, the polynomial invariants provide a non-Hausdorff separation of the orbit space $\mathbb{R}^{1,1}/\Gamma$.
	
	\begin{proposition} \label{prop:separation}
		Let $\Gamma = \langle \mathcal{H}_\beta \rangle$ be the discrete subgroup of $O(1,1)$ generated by a hyperbolic rotation. The map $\rho: \mathbb{R}^{1,1} \to \mathbb{R}$ does not separate $\Gamma$-orbits; specifically, the induced map on the orbit space $\bar{\rho} : \mathbb{R}^{1,1}/\Gamma \to \mathbb{R}$ is not injective.
	\end{proposition}
	\begin{proof}
		Consider the points $p = (2, 1)$ and $q = (-2, 1)$ in $\mathbb{R}^{1,1}$. Both satisfy $\rho(p) = \rho(q) = 3$, yet they belong to distinct $\Gamma$-orbits. In light-cone coordinates $(u, v)$, these points are given by $p = (3, 1)$ and $q = (-1, -3)$. Since the action $(u, v) \mapsto (u e^{n\beta}, v e^{-n\beta})$ preserves the quadrants of the $uv$-plane, the strictly positive coordinates of $p$ and the strictly negative coordinates of $q$ ensure that no $g \in \Gamma$ can map one point to the other. Thus, $\bar{\rho}(p) = \bar{\rho}(q)$ while $p \neq q$ in $\mathbb{R}^{1,1}/\Gamma$.
	\end{proof}
	
	This lack of injectivity implies that the orbit space is not homeomorphic to the image of $\rho$. In the compact case, the orbit space is uniquely characterized as a semi-algebraic set by the polynomial invariants; here, the polynomial ring $\mathcal{P}(\mathbb{R}^{1,1})^\Gamma$ lacks the resolution to distinguish between the disjoint components of the invariant manifolds.

\bigskip

\subsection{Third obstruction: Failure of the core analytical machinery}

The proof of Schwarz's Theorem for compact groups (\cite{schwarz1975}) relies on a analytical machinery that is unavailable in the non-compact setting. The following obstructions represent fundamental barriers to adapting the classical framework to discrete Lorentz groups:

\begin{itemize}
	\item \textbf{Absence of an averaging operator:} For a compact group $G$, the Haar integral provides a continuous projection $\operatorname{Av}: \mathcal{E}(\mathbb{R}^n) \to \mathcal{E}(\mathbb{R}^n)^G$. In the case of an infinite discrete group such as $\Gamma$, the lack of a finite invariant measure precludes the existence of such an operator, eliminating the standard method for constructing invariant smooth maps.
	\item \textbf{Non-properness of the Hilbert map:} Unlike the compact case, where the map $\rho$ is proper, the Lorentzian quadratic form fails this condition. 
\end{itemize}

\begin{proposition} \label{prop:nonproper}
	The Hilbert map $\rho: \mathbb{R}^{1,1} \to \mathbb{R}$ is not a proper map.
\end{proposition}
\begin{proof}
	The preimage $\rho^{-1}(K)$ is a closed region in $\mathbb{R}^{1,1}$ bounded by hyperbolas. Since these branches extend asymptotically toward the null lines $x = \pm y$, the set $\rho^{-1}(K)$ is unbounded and, consequently, non-compact.
\end{proof}

\begin{itemize}
	\item \textbf{Non-Hausdorff orbit space topology:} The failure of orbit separation (Proposition~\ref{prop:separation}) implies that the quotient space $\mathbb{R}^{1,1}/\Gamma$ is non-Hausdorff. In Schwarz’s original framework, the orbit space is required to be a well-behaved (stratified) object where smooth invariants can be uniquely identified. In the Lorentzian case, the existence of distinct orbits whose closures intersect in the manifold topology prevents the quotient from inheriting a standard smooth or even Hausdorff structure.
\end{itemize}

These obstructions collectively demonstrate that the correspondence between algebraic and smooth invariants is fundamentally linked to the compactness of the group. The transition to the Lorentzian setting introduces pathologies that preclude any direct application of the classical Schwarz framework.

\section{Cocompact actions and the Hilbert-Weyl theorem}

The Hilbert-Weyl Theorem establishes that for a compact Lie group $\Gamma$ acting linearly on $\mathbb{R}^n$, the ring of invariant polynomials $\mathcal{P}[\mathbb{R}^n]^\Gamma$ is finitely generated. Given the extension of certain aspects of invariant theory to non-compact Lorentz groups in Theorem A, it is natural to investigate whether the Hilbert–Weyl theorem persists for other classes of non-compact actions, specifically cocompact actions.

In this section, we demonstrate that the Hilbert-Weyl theorem fails for cocompact actions. This failure arises from an incompatibility between polynomial growth and the geometric dynamics of such actions. While the polynomial invariant theory is restricted to constants, the smooth invariant theory admits a complete characterization. This divergence highlights a significant departure from the compact case.

\subsection{The algebraic obstruction: Periodicity vs. polynomial growth}

Recall that a discrete group $\Gamma$ acting on $\mathbb{R}^n$ is said to act cocompactly if the quotient space $\mathbb{R}^n/\Gamma$ is compact. When the action is also free and proper, this quotient $M = \mathbb{R}^n/\Gamma$ inherits the structure of a compact smooth manifold; a prototypical example is the $n$-torus $\mathbb{T}^n$ arising from the translation action of $\mathbb{Z}^n$.

\begin{theoremB}
Let \( \Gamma \) be a discrete group acting freely, properly, and cocompactly on \( \mathbb{R}^n \). Then the algebra of polynomial invariants is trivial, i.e., \( P(\mathbb{R}^n)^\Gamma = \mathbb{R} \).
\end{theoremB}

\begin{proof}
	Suppose there exists a non-constant invariant polynomial $p \in \mathcal{P}(\mathbb{R}^n)^\Gamma$. Since $\Gamma$ acts cocompactly and properly, the quotient $M = \mathbb{R}^n/\Gamma$ is a compact Hausdorff space. By $\Gamma$-invariance, $p$ descends to a continuous function $\tilde{p}: M \to \mathbb{R}$ such that $p = \tilde{p} \circ \pi$, where $\pi: \mathbb{R}^n \to M$ is the quotient projection. Since $M$ is compact, $\tilde{p}$ is bounded. The surjectivity of $\pi$ then implies that $p$ is bounded on $\mathbb{R}^n$. However, a non-constant polynomial on $\mathbb{R}^n$ is necessarily unbounded. It follows that every $\Gamma$-invariant polynomial must be constant.
\end{proof}

The observation in Theorem B reveals a general obstruction: cocompactness forces polynomial invariants to degenerate to constants. This phenomenon highlights a divergence from the compact case, where the existence of a finite generating set is guaranteed.

\begin{remark}
	The requirement of properness ensures that the quotient space $\mathbb{R}^n/\Gamma$ is Hausdorff. Without this condition, an invariant polynomial $p$ might not descend to a well-defined continuous function on the quotient. Taken together, the conditions of freeness, properness, and cocompactness ensure that $\mathbb{R}^n/\Gamma$ is a compact smooth manifold and that the quotient projection $\pi$ is a covering map.
\end{remark}

The behavior of smooth invariants in the cocompact setting contrasts with the algebraic collapse described in Theorem B. To illustrate this, consider the translation action of $\Gamma = \mathbb{Z}^n$ on $\mathbb{R}^n$.
\begin{itemize}
		\item \textbf{Action}: For $k \in \mathbb{Z}^n$, the action is given by $(k_1, \dots, k_n) \cdot (x_1, \dots, x_n) = (x_1 + k_1, \dots, x_n + k_n)$.
		\item \textbf{Geometry}: This action is free and proper, yielding the $n$-torus $\mathbb{T}^n = \mathbb{R}^n/\mathbb{Z}^n$ as a compact quotient manifold.
		\item \textbf{Invariants}: Any invariant polynomial $p \in \mathcal{P}(\mathbb{R}^n)^{\mathbb{Z}^n}$ must be periodic ($p(x+k) = p(x)$), which forces $p$ to be constant. Conversely, in the smooth category, the space of invariants $\mathcal{E}(\mathbb{R}^n)^{\mathbb{Z}^n}$ is isomorphic to $\mathcal{E}(\mathbb{T}^n)$.
\end{itemize}

This example indicates that while the algebra of polynomial invariants is restricted to constants, the smooth invariants encompass the entire function theory of the compact quotient manifold.

\subsection{A Schwarz-type theorem for cocompact actions}

While polynomial invariant theory reduces to constants for cocompact actions, the smooth invariant theory admits a complete characterization. The following theorem provides a smooth analogue of Schwarz's theorem for the cocompact setting, demonstrating that the space of smooth invariants is determined by a finite collection of smooth maps.

\begin{theoremC}[Smooth Invariants for Cocompact Actions]
	Let \( \Gamma \) be a discrete group acting freely, properly, and cocompactly on a smooth manifold \( M \). Let \( \pi: M \to M/\Gamma \) be the projection map.
	Then:
	\begin{enumerate}
		\item The pullback
		\[
		\pi^*: \mathcal{E}(M/\Gamma) \xrightarrow{\cong} \mathcal{E}(M)^\Gamma
		\]
		is a topological isomorphism of Fréchet spaces under the standard $ C^\infty $ topology.
		\item There exist smooth invariants \( \sigma_1, \ldots, \sigma_m \in \mathcal{E}(M)^\Gamma \) such that
		\[
		\sigma^* \mathcal{E}(\mathbb{R}^m) = \mathcal{E}(M)^\Gamma,
		\]
		where \( \sigma = (\sigma_1, \ldots, \sigma_m): M \to \mathbb{R}^m \).
		\item For any manifold \( X \) with trivial \( \Gamma \)-action,
		\[
		(\mathrm{id} \times \sigma)^* \mathcal{E}(X \times \mathbb{R}^m) = \mathcal{E}(X \times M)^\Gamma.
		\]
	\end{enumerate}
	
\end{theoremC}

	We establish two preliminary results to support the proof.
	
	\begin{lemma} \label{lem:quotient_manifold}
	Under the hypotheses of Theorem C, the quotient space \( M/\Gamma \) is a compact smooth manifold and the projection \( \pi: M \to M/\Gamma \) is a smooth covering map.
	\end{lemma}
	\begin{proof}
	Since the action is free and proper, the quotient $ M/\Gamma $ admits a unique smooth structure such that $ \pi $ is a smooth covering map \cite[Theorem 7.10]{Lee2012}. The compactness of $ M/\Gamma $ follows directly from the cocompactness of the action.
	\end{proof}
	
	\begin{lemma} \label{lem:pi_continuous}
		The pullback \( \pi^*: \mathcal{E}(M/\Gamma) \to \mathcal{E}(M)^\Gamma \) is continuous and injective.
	\end{lemma}
	\begin{proof}
		Continuity follows immediately from the definition of the \( C^\infty \) topology. Injectivity is a consequence of the surjectivity of $ \pi $.
	\end{proof}
	
	We now prove each part of the Theorem C.
\begin{proof}[Proof of Theorem C]
	\begin{enumerate}
	   \item By Lemma \ref{lem:quotient_manifold}, $M/\Gamma$ is a compact manifold. 
	   For any $f \in \mathcal{E}(M)^\Gamma$, let $\tilde{f}: M/\Gamma \to \mathbb{R}$ be the map $\tilde{f}([p]) = f(p)$, which is well-defined by $\Gamma$-invariance.
	   Since the action is free and proper, $\pi: M \to M/\Gamma$ is a smooth covering map. 
	   For any $[p] \in M/\Gamma$, there exists a neighborhood $U$ of $[p]$ and a smooth local section $s: U \to M$ such that $\pi \circ s = \mathrm{id}_U$. 
	   On $U$, we have $\tilde{f}|_U = f \circ s$, which is smooth as a composition of smooth maps. 
	   Since such neighborhoods cover $M/\Gamma$, $\tilde{f}$ is smooth and satisfies $\pi^*\tilde{f} = f$.
	   Thus, $\pi^*$ is a continuous linear bijection between Fréchet spaces, where continuity and injectivity are established by Lemma~\ref{lem:pi_continuous}. By the Open Mapping Theorem for Fréchet spaces, the inverse map $(\pi^*)^{-1}$ is also continuous. It follows that $\pi^*: \mathcal{E}(M/\Gamma) \to \mathcal{E}(M)^\Gamma$ is a topological isomorphism.
		
		\item Since $M/\Gamma$ is compact, Whitney's Embedding Theorem (\cite[Theorem 6.14]{Lee2012}) guarantees the existence of a smooth embedding $j: M/\Gamma \hookrightarrow \mathbb{R}^m$ for some $m \in \mathbb{N}$. Let $j = (j_1, \ldots, j_m)$ and define the invariants $\sigma_i = \pi^* j_i \in \mathcal{E}(M)^\Gamma$. Set $\sigma = (\sigma_1, \ldots, \sigma_m): M \to \mathbb{R}^m$.
		Let $f \in \mathcal{E}(M)^\Gamma$. By part (1), there exists a unique $\tilde{f} \in \mathcal{E}(M/\Gamma)$ such that $f = \pi^* \tilde{f}$. Because $j$ is a smooth embedding of a compact manifold, the Tubular Neighborhood Theorem (\cite[Theorems 6.17 and 6.18]{Lee2012}) provides an open neighborhood $W$ of $j(M/\Gamma)$ in $\mathbb{R}^m$ and a smooth retraction $r: W \to j(M/\Gamma)$.
		Define $F_0 = \tilde{f} \circ j^{-1} \circ r$ on $W$. Since $j(M/\Gamma)$ is closed in $\mathbb{R}^m$, there exists a smooth bump function $\rho \in \mathcal{E}(\mathbb{R}^m)$ supported in $W$ such that $\rho \equiv 1$ on an open neighborhood of $j(M/\Gamma)$. Then $F = \rho \cdot F_0$ (extended by zero outside $W$) belongs to $\mathcal{E}(\mathbb{R}^m)$ and satisfies $F|_{j(M/\Gamma)} = \tilde{f} \circ j^{-1}$. We then verify that $f = \sigma^* F$:
		\[\begin{aligned}
			(\sigma^* F)(p) &= F(\sigma(p)) = F(j(\pi(p))) \\
			&= (\tilde{f} \circ j^{-1})(j(\pi(p))) \\
			&= \tilde{f}(\pi(p)) = f(p).
		\end{aligned}
		\]
		Thus $f \in \sigma^* \mathcal{E}(\mathbb{R}^m)$, which completes the proof.

		\item Let $f \in \mathcal{E}(X \times M)^\Gamma$, where $\Gamma$ acts trivially on $X$. The $\Gamma$-invariance of $f$ in the $M$-variable ensures that $\tilde{f}(x, [p]) = f(x, p)$ is a well-defined map on $X \times M/\Gamma$. To see that $\tilde{f}$ is smooth, note that for any $(x_0, [p_0]) \in X \times M/\Gamma$, there exist neighborhoods $U \subset X$ and $V \subset M/\Gamma$ and a smooth local section $s: V \to M$ such that $\tilde{f}|_{U \times V} = f \circ (\mathrm{id}_X \times s)$, which is smooth.Let $j: M/\Gamma \hookrightarrow \mathbb{R}^m$ be the embedding from part (2) and define $k = \mathrm{id}_X \times j: X \times M/\Gamma \hookrightarrow X \times \mathbb{R}^m$. Since $M/\Gamma$ is compact, $X \times M/\Gamma$ is a closed submanifold of $X \times \mathbb{R}^m$. The tubular neighborhood argument used in part (2) extends to this setting (parameterized by $X$), providing a smooth extension $F \in \mathcal{E}(X \times \mathbb{R}^m)$ such that $\tilde{f} = F \circ k$. We verify that $f = (\mathrm{id}_X \times \sigma)^* F$:
		\[\begin{aligned}
			f(x, p) &= \tilde{f}(x, \pi(p)) \\
			&= F(k(x, \pi(p))) \\
			&= F(x, j(\pi(p))) \\
			&= F(x, \sigma(p)) \\
			&= ((\mathrm{id}_X \times \sigma)^* F)(x, p).
		\end{aligned}
		\]
		Thus $f \in (\mathrm{id}_X \times \sigma)^* \mathcal{E}(X \times \mathbb{R}^m)$, which completes the proof.
	\end{enumerate}
\end{proof}

\begin{remark}
	Theorem~C identifies a distinction between the algebraic and smooth regimes: while polynomial invariant theory reduces to constants for cocompact actions (Theorem~B), the smooth invariant theory is preserved and admits a characterization via the geometry of the quotient manifold.
\end{remark}

To illustrate the implications of Theorem~C, we examine the case of the $n$-torus, where the action is given by integer translations on $\mathbb{R}^n$.

\begin{example}[The $n$-torus]
	Let $\Gamma = \mathbb{Z}^n$ act on $M = \mathbb{R}^n$ by translations. This action is free, proper, and cocompact, with the $n$-torus $\mathbb{T}^n$ as the quotient manifold.
	
	\begin{itemize}
		\item \textbf{Part (1)}: The space of invariants $\mathcal{E}(\mathbb{R}^n)^{\mathbb{Z}^n}$ is topologically isomorphic to $\mathcal{E}(\mathbb{T}^n)$.
		\item \textbf{Part (2)}: A smooth embedding $j: \mathbb{T}^n \hookrightarrow \mathbb{R}^{2n}$ provides invariants $\sigma_1, \ldots, \sigma_{2n}$ such that
		\[
		\mathcal{E}(\mathbb{R}^n)^{\mathbb{Z}^n} = \sigma^*\mathcal{E}(\mathbb{R}^{2n}).
		\]
		\item \textbf{Part (3)}: For any manifold $X$, every $\mathbb{Z}^n$-invariant smooth function on $X \times \mathbb{R}^n$ can be expressed as a smooth function of the invariants $\sigma_i$ and the coordinates of $X$.
	\end{itemize}
	
This example shows that Theorem~C realizes the algebra of smooth invariants as the full function theory of the compact quotient, whereas the polynomial invariants consist only of scalars.
\end{example}

The hypotheses of Theorem~C are satisfied by a broad class of geometric actions beyond translations. Notable cases include:
\begin{itemize}
		\item \textbf{Torsion-free crystallographic groups}: These act properly discontinuously and cocompactly on $\mathbb{R}^n$ by isometries, yielding compact flat manifolds as quotients.
		\item \textbf{Torsion-free cocompact Fuchsian groups}: Discrete subgroups of $\mathrm{PSL}(2,\mathbb{R})$ acting freely on the hyperbolic plane $\mathbb{H}^2$, resulting in compact Riemann surfaces of genus $g \geq 2$.
		\item \textbf{Nilmanifolds}: Actions of discrete cocompact subgroups $\Gamma$ on a nilpotent Lie group $N$ (e.g., the Heisenberg group), where the quotient $N/\Gamma$ is a compact nilmanifold.
\end{itemize}

In each instance, the theorem ensures that the smooth invariant theory is entirely determined by the smooth structure of the respective compact quotient.

\section{Synthesis: A comparative analysis of invariant regimes}

The results established in the previous sections, particularly Theorems A, B, and C, reveal that the behavior of invariant functions is deeply contingent upon the geometric and topological properties of the group action. By situating these findings within the broader context of classical invariant theory, we identify four distinct regimes. These regimes are distinguished by the presence or absence of finite generation for polynomial and smooth invariants, as summarized in Table~\ref{tab:comparison}.

\begin{table}[ht]
	\centering
	\small
	\renewcommand{\arraystretch}{1.3}
	\begin{tabular}{@{}llll@{}}
		\toprule
		\textbf{Regime} & \textbf{Polynomial Invariants} & \textbf{Smooth Invariants} & \textbf{Defining Characteristic} \\
		\midrule
		\textbf{Compact Groups} & Finitely generated & Generated by $\mathcal{P}(M)^\Gamma$ & Haar averaging projection \\
		(Schwarz, Weyl)         & (Hilbert--Weyl)    & (Schwarz Theorem)                  & \\
		\addlinespace
		\textbf{Reductive Groups} & Finitely generated & Generated by $\mathcal{P}(M)^\Gamma$ & Reynolds operator \\
		(Luna)                    & (Hilbert)          & (Luna's Theorem)                   & \\
		\addlinespace
		\textbf{Cocompact Actions} & \textbf{Trivial} & Finitely generated by $\mathcal{E}$ & Geometric descent to \\
		(This work)                & (Theorem B)      & (Theorem C)                         & compact quotient $M/\Gamma$ \\
		\addlinespace
		\textbf{Discrete Lorentz}  & Finitely generated & \textbf{Not} generated by $\mathcal{P}$ & Orbit space pathology; \\
		(This work)                & (Theorem A)      & (Schwarz failure)                   & non-properness \\
		\bottomrule
	\end{tabular}
	\caption{Comparison of polynomial and smooth invariant theory across different group action regimes.}
	\label{tab:comparison}
\end{table}

This comparative framework clarifies that the validity of the classical Schwarz property, where the algebra of smooth invariants is generated by its polynomial subring, is not a generic consequence of the finite generation of $\mathcal{P}(M)^\Gamma$. Rather, it depends on structural mechanisms, such as averaging operators or properness of the action, which are not universally preserved in non-compact settings.

\section{Conclusion and perspectives}

This work demonstrates that the correspondence between algebraic and smooth invariants depends strictly on the properness of the group action and the geometry of the quotient. We have shown that the Schwarz property, the generation of smooth invariants by polynomial ones, is inhibited by the topological pathologies of non-proper discrete Lorentz actions, even when the underlying polynomial ring remains finitely generated. Conversely, for cocompact actions, we established that while the polynomial theory reduces to constants, a complete smooth theory is recovered through the manifold structure of the quotient.

These results contribute to the study of invariant theory for non-reductive groups, a field where the standard tools of Geometric Invariant Theory (GIT) often require refinement. Our findings suggest that the stability of the orbit space is a necessary condition for the alignment of algebraic and analytic descriptions of symmetry.

A direct application of this framework lies in the extension to smooth equivariant maps. Since the generation of equivariant sections (Poénaru’s theorem) typically relies on the existence of a smooth invariant basis, the obstructions identified in this paper provide a boundary for the extension of equivariant results to non-compact regimes. This remains a subject for future investigation, aiming to establish the precise conditions under which equivariant symmetry can be reduced to the geometry of the quotient.

This work aligns with the modern program of extending Hilbert's 14th problem to the smooth category, demonstrating that the existence of a categorical quotient in the sense of GIT does not imply the existence of a well-behaved analytic quotient when properness is lost.

\section*{Conflict of interest statement}

The author declares that there is no conflict to report.

\section*{Data availability statement}
No datasets were generated or analysed during the current study. This is a theoretical work in pure mathematics, and all results are proven within the article.

\section*{Acknowledgements}

This study was financed, in part, by the São Paulo Research Foundation (FAPESP), Brasil. Process Number \#2022/12906-3.

\bibliographystyle{plain}
\bibliography{referencias}

\end{document}